\documentclass[12pt,a4paper]{article}
\usepackage{mathrsfs}
\usepackage{epsf, graphicx}
\usepackage{latexsym,amsfonts,amsbsy,amssymb}
\usepackage{amsmath,amsthm}
\makeatletter 

\@addtoreset{equation}{section}
\makeatother 
\allowdisplaybreaks 

\newtheorem{theorem}{Theorem}[section]
\newtheorem{lemma}{Lemma}[section]

\newtheorem{algorithm}{Algorithm}[section]

\newtheorem{proposition}{Proposition}[section]

\begin{document}
\title{A Multi-level Correction Scheme for Eigenvalue Problems 
\footnote{This work is supported in part by National Science Foundation of China (NSFC 11001259, 11031006, 2011CB309703).}}

\author{Qun Lin\footnote{LSEC, ICMSEC, Academy of Mathematics and Systems
Science, Chinese Academy of Sciences,  Beijing 100190, China
(linq@lsec.cc.ac.cn)} \   and\
Hehu Xie\footnote{LSEC, ICMSEC, Academy of Mathematics and Systems
Science, Chinese Academy of Sciences,  Beijing 100190, China
(hhxie@lsec.cc.ac.cn) } }
\date{Jan 23 2011 }
\maketitle

\thanks{}

\begin{abstract}
In this paper, a new type of multi-level correction scheme is
proposed for solving eigenvalue problems by finite element method.
With this new scheme, the accuracy of eigenpair approximations  can
be improved after each correction step which only needs to solve a
source problem on finer finite element space and an eigenvalue
problem on the coarsest finite element space.  This correction
scheme can improve the efficiency of solving eigenvalue problems by
finite element method.




\end{abstract}

\section{Introduction}
The purpose of this paper is to propose a new type of multi-level
correction scheme based on finite element discretization to solve
eigenvalue problems. The two-grid method for solving eigenvalue
problems has been proposed and analyzed by Xu and Zhou in
\cite{XuZhou}.  The idea of the two-grid comes from \cite{Xu_Two_Grid, Xu_Nonlinear}
for nonsymmetric or indefinite problems and nonlinear elliptic equations. Since then,
there have existed many numerical
methods for solving eigenvalue problems based on the idea of
two-grid method (\cite{AndreevLazarovRacheva,ChenJiaXie,RachevaAndreev}).

In this paper, we present a new type of multi-level correction
scheme for solving eigenvalue problems. With the new proposed method
solving eigenvalue problem will not be much more difficult than the
solution of the corresponding source problem. Our method is some
type of operator iterative method (\cite{Lin,XuZhou,Zhou}). The
correction method for eigenvalue problem in this paper is based on a
series of finite element spaces with different approximation
properties which are related to the multilevel method (\cite{Xu}).

The standard Galerkin finite element method for eigenvalue problem
has been extensively investigated, e.g. Babu\v{s}ka and Osborn
\cite{Babuska2,BabuskaOsborn}, Chatelin \cite{Chatelin} and
references cited therein. Here we adopt some basic results in these
papers in our analysis. The finite element method for eigenvalue
problem has been developed
 well and many high efficient methods have also been proposed and analyzed for different types of eigenvalue problems
 (\cite{ChenJiaXie,ChenLin,JiaXieYinGao,LinLin,LinLu,LinYan,RachevaAndreev,XuZhou}).

The corresponding error estimates of the type of multi-level
correction scheme which is introduced here will be analyzed. Based
on the analysis, the new method can improve the convergence rate of
the eigenpair approximations after each correction step. The
multi-level correction procedure can be described as follows: (1)\
solve the eigenvalue problem in the coarsest finite element space;
(2)\ solve an additional source problem in an augmented space using
the previous obtained eigenvalue multiplying the corresponding
eigenfunction as the load vector; (3)\ solve eigenvalue problem
again on the finite element space which is constructed by combining
the coarsest finite element space with the obtained eigenfunction
approximation in step (2). Then go to step (2) for next loop.

Similarly to \cite{XuZhou}, in order to describe our method clearly,
we give the following simple Laplace eigenvalue problem to
illustrate the main idea in this paper with multi-grid
implementation way (see section 5).

Find$(\lambda,u)$ such that
\begin{equation}\label{problem}
\left\{
\begin{array}{rcl}
-\Delta u&=&\lambda u,\quad{\rm in}\ \Omega,\\
u&=&0,\quad\ \  {\rm on}\ \partial\Omega,\\
\int_{\Omega} u^2d\Omega&=&1,
\end{array}
\right.
\end{equation}
where $\Omega\subset\mathcal{R}^2$ is a bounded domain with
Lipschitz boundary $\partial\Omega$ and $\Delta$ denote the Laplace
operator.

Let $V_H$ denote the coarsest linear finite element defined on the
coarsest mesh $\mathcal{T}_H$. Additionally, we also need to
construct a series of finite element spaces $V_{h_2}$, $ V_{h_3}$,
$\cdots$, $V_{h_n}$ which are defined on the corresponding series of
meshes $\mathcal{T}_{h_k}\ (k=2,3,\cdots n)$ such that $V_H\subset
V_{h_2}\subset \cdots \subset V_{h_n}$
(\cite{BrennerScott,Ciarlet}). Our multi-level correction algorithm
to obtain the approximation of the eigenpair can be defined as
follows (see section 3 and section 4):

\begin{enumerate}
\item Solve an eigenvalue problem in the coarsest space $V_H$:

Find $(\lambda_H, u_H)\in \mathcal{R}\times V_H$ such that
$\|u_H\|_0=1$ and
\begin{eqnarray*}
\int_{\Omega}\nabla u_H\nabla v_Hd\Omega &=&
\lambda_H\int_{\Omega}u_Hv_Hd\Omega,\ \ \ \forall v_H\in V_H.
\end{eqnarray*}
\item Set $h_1=H$ and Do $k=1,\cdots, n-2$
\begin{itemize}
\item
Solve the following auxiliary source problem:

Find $\widetilde{u}_{h_{k+1}}\in V_{h_{k+1}}$ such that
\begin{eqnarray*}
\int_{\Omega}\nabla \widetilde{u}_{h_{k+1}}\nabla v_{h_{k+1}}d\Omega
&=& \lambda_{h_k}\int_{\Omega}u_{h_k}v_{h_{k+1}}d\Omega,\ \ \
\forall v_{h_{k+1}}\in V_{h_{k+1}}.
\end{eqnarray*}
\item Define a new finite element space $V_{H,h_{k+1}}=V_H+{\rm
span}\{\widetilde{u}_{h_{k+1}}\}$ and solve the following eigenvalue
problem:

Find $(\lambda_{h_{k+1}}, u_{h_{k+1}})\in \mathcal{R}\times
V_{H,h_{k+1}}$ such that $\|u_{h_{k+1}}\|_0=1$ and
\begin{eqnarray*}
\int_{\Omega}\nabla u_{h_{k+1}}\nabla v_{H,h_{k+1}}d\Omega
=\lambda_{h_{k+1}}\int_{\Omega} u_{h_{k+1}} v_{H,h_{k+1}}d\Omega, \
\ \ \forall v_{H,h_{k+1}}\in V_{H,h_{k+1}}.
\end{eqnarray*}
\end{itemize}
end Do
\item
Solve the following auxiliary source problem:

Find $\widetilde{u}_{h_{n}}\in V_{h_{n}}$ such that
\begin{eqnarray*}
\int_{\Omega}\nabla u_{h_n}\nabla v_{h_n}d\Omega &=&
\lambda_{h_{n-1}}\int_{\Omega} u_{h_{n-1}}v_{h_{n}}d\Omega,\ \ \
\forall v_{h_n}\in V_{h_n}.
\end{eqnarray*}
Then compute the Rayleigh quotient
\begin{eqnarray*}
\lambda_{h_n}=\frac{\|\nabla u_{h_n}\|_0^2}{\|u_{h_n}\|_0^2}.
\end{eqnarray*}
\end{enumerate}
If, for example, $\lambda_H$ is the first eigenvalue of the
 problem at the first step and $\Omega$ is a convex domain, then we can establish the
following results (see section 3 and section 4 for details)
\begin{eqnarray*}
\|\nabla (u-u_{h_n})\|_0
=\mathcal{O}\Big(\sum_{k=1}^nh_kH^{n-k}\Big),\ \ \ {\rm and}\ \ \
|\lambda_{h_n}-\lambda| =
\mathcal{O}\Big(\sum_{k=1}^nh_k^2H^{2(n-k)}\Big).
\end{eqnarray*}
These two estimates means that we can obtain asymptotic optimal
errors by taking $H=\sqrt[n]{h_n}$ and $h_k=H^k\ (k=2,\cdots,n-1)$.

In this method, we replace solving eigenvalue problem on the finest
finite element space by solving a series of boundary value problems
in the corresponding series of finite element spaces and a series of
eigenvalue problems in the coarsest finite element space. As we
know, there exists multigrid method for solving boundary value
problems efficiently. So this correction method can improve the
efficiency of solving eigenvalue problems.

An outline of the paper goes as follows. In Section 2, we introduce
finite element method for eigenvalue problem and the corresponding
error estimates. A type of one correction step is given in section
3. In Section 4, we propose a type of multi-level correction
algorithm for solving eigenvalue problem by finite element methods.
In section 5, two numerical examples are presented to validate our
theoretical analysis and some concluding remarks are given in the
last section.

\section{Discretization by finite element method}
In this section, we introduce some notation and error estimates of
the finite element approximation for eigenvalue problems. In this
paper, the letter $C$ (with or without subscripts) denotes a generic
positive constant which may be different at different occurrences.
For convenience, the symbols $\lesssim$, $\gtrsim$ and $\approx$
will be used in this paper. That $x_1\lesssim y_1, x_2\gtrsim y_2$
and $x_3\approx y_3$, mean that $x_1\leq C_1y_1$, $x_2 \geq c_2y_2$
and $c_3x_3\leq y_3\leq C_3x_3$ for some constants $C_1, c_2, c_3$
and $C_3$ that are independent of mesh sizes.

Let $(V, \|\cdot\|)$ be a real Hilbert space with inner product
$(\cdot, \cdot)$ and norm $\|\cdot\|$, respectively. Let
$a(\cdot,\cdot)$, $b(\cdot,\cdot)$ be two symmetric bilinear forms
on $X \times X$ satisfying
\begin{eqnarray}
a(w, v) &\lesssim&\|w\|\|v\|, \ \ \forall w\in V\ {\rm and}\ \forall v\in V,\label{Bounded}\\
\|w\|^2&\lesssim& a(w,w),\ \ \forall w\in V\ {\rm and}\ 0<b(w,w), \
\ \forall w\in V\ {\rm and}\ w\neq 0.\label{Coercive}
\end{eqnarray}
From (\ref{Bounded}) and (\ref{Coercive}), we know that
$\|\cdot\|_a:=a(\cdot,\cdot)^{1/2}$ and $\|\cdot\|$ are two
equivalent norms on $V$. We assume that the norm $\|\cdot\|$ is
relatively compact with respect to the norm
$\|\cdot\|_b:=b(\cdot,\cdot)^{1/2}$. We shall use $a(\cdot,\cdot)$
and $\|\cdot\|_a$ as the inner product and norm on $V$ in the rest
of this paper.

 Set $$W := {\rm the\ completion\ of}\ V \ {\rm with\
respect\ to}\ \|\cdot\|_b.$$ Thus $W$ is a Hilbert space with the
inner product $b(\cdot, \cdot)$ and compactly imbedded in $V$.
Construct a ``negative space" by $V' = {\rm the\ dual\ of}\ V\ {\rm
with\ a\ norm}\  \|\cdot\|_{-a}$ given by
\begin{eqnarray}
\|w\|_{-a}=\sup_{v\in V,\|v\|_a=1}b(w,v).
\end{eqnarray}
Then $W \subset V'$ compactly, and for $v \in V$, $b(w, v)$ has a
continuous extension to $w\in V'$ such that $b(w, v)$ is continuous
on $V'$ by Hahn-Banach theorem (\cite{Conway}). We assume that
$V_h\subset V$ is a family of finite-dimensional spaces that satisfy
the following assumption:

For any $w \in V$
\begin{eqnarray}\label{Approximation_Property}
\lim_{h\rightarrow0}\inf_{v\in V_h}\|w-v\|_a = 0.
\end{eqnarray}
Let $P_h$ be the finite element projection operator of $V$ onto
$V_h$ defined by
\begin{eqnarray}\label{Projection_Problem}
a(w - P_hw, v) = 0,\ \ \ \forall w\in V\ {\rm and}\ \forall v\in
V_h.
\end{eqnarray}
Obviously
\begin{eqnarray}\label{Projection_Optimal}
\|P_hw\|_a\leq \|w\|_a,\ \ \ \forall w\in V.
\end{eqnarray}
For any $w\in V$, by (\ref{Approximation_Property}) we have
\begin{eqnarray}
\|w-P_hw\|_a&=&o(1),\ \ \ {\rm as}\ h\rightarrow 0.
\end{eqnarray}
Define $\eta_a(h)$ as
\begin{eqnarray}
\eta_a(h)=\sup_{f\in V,\|f\|_a=1}\inf_{v\in V_h}\|T f-v\|_a,
\end{eqnarray}
where the operator $T: V'\mapsto V$ is defined as
\begin{eqnarray}
a(Tf,v)&=&b(f,v),\ \ \ \forall f\in V'\ {\rm and}\ \forall v\in V.
\end{eqnarray}
In order to derive the error estimate of eigenpair approximation in
negative norm $\|\cdot\|_{-a}$, we need the following negative norm
error estimate of the finite element projection operator $P_h$.
\begin{lemma}\label{Negative_norm_estimate_Lemma}
(\cite[Lemma 3.3 and Lemma 3.4]{BabuskaOsborn})
\begin{eqnarray}
\eta_a(h)&=&o(1),\ \ \ {\rm as}\ h\rightarrow 0,
\end{eqnarray}
and
\begin{eqnarray}\label{Negative_norm_Error}
\|w-P_hw\|_{-a}&\lesssim&\eta_a(h)\|w-P_hw\|_a,\ \ \ \forall w\in V.
\end{eqnarray}
\end{lemma}

In our methodology description, we are concerned with the following
general eigenvalue problem:

Find $(\lambda, u )\in \mathcal{R}\times V$ such that $b(u,u)=1$ and
\begin{eqnarray}
a(u,v)&=&\lambda b(u,v),\quad \forall v\in V. \label{weak_problem}
\end{eqnarray}

For the eigenvalue $\lambda$, there exists the following Rayleigh
quotient expression (\cite{Babuska2,BabuskaOsborn,XuZhou})
\begin{eqnarray}\label{Rayleigh_quotient}
 \lambda=\frac{a(u,u)}{b(u,u)}.
\end{eqnarray}
From \cite{BabuskaOsborn,Chatelin}, we know the eigenvalue problem
(\ref{weak_problem}) has an eigenvalue sequence $\{\lambda_j \}:$
$$0<\lambda_1\leq \lambda_2\leq\cdots\leq\lambda_k\leq\cdots,\ \ \
\lim_{k\rightarrow\infty}\lambda_k=\infty,$$ and the associated
eigenfunctions
$$u_1,u_2,\cdots,u_k,\cdots,$$
where $b(u_i,u_j)=\delta_{ij}$. In the sequence $\{\lambda_j\}$, the
$\lambda_j$ are repeated according to their geometric multiplicity.

Now, let us define the finite element approximations of the problem
(\ref{weak_problem}). First we generate a shape-regular
decomposition of the computing domain $\Omega\subset \mathcal{R}^d\
(d=2,3)$ into triangles or rectangles for $d=2$ (tetrahedrons or
hexahedrons for $d=3$). The diameter of a cell $K\in\mathcal{T}_h$
is denoted by $h_K$. The mesh diameter $h$ describes the maximum
diameter of all cells $K\in\mathcal{T}_h$. Based on the mesh
$\mathcal{T}_h$, we can construct a finite element space denoted by
$V_h\subset V$. In order to do multi-level correction method, we
start the process on the original mesh $\mathcal{T}_H$ with the mesh
size $H$ and the original coarsest finite element space $V_H$
defined on the mesh $\mathcal{T}_H$.

Then we can define the approximation of eigenpair $(\lambda,u)$ of
(\ref{weak_problem}) by the finite element method as:

Find $(\lambda_h, u_h)\in \mathcal{R}\times V_h$ such that
$b(u_h,u_h)=1$ and
\begin{eqnarray}\label{weak_problem_Discrete}
a(u_h,v_h)&=&\lambda_hb(u_h,v_h),\quad\ \  \ \forall v_h\in V_h.
\end{eqnarray}

From (\ref{weak_problem_Discrete}), we can know the following
Rayleigh quotient expression for $\lambda_h$ holds
(\cite{Babuska2,BabuskaOsborn,XuZhou})
\begin{eqnarray}\label{eigenvalue_Rayleigh}
\lambda_h &=&\frac{a(u_h,u_h)}{b(u_h,u_h)}.
\end{eqnarray}
Similarly, we know from \cite{BabuskaOsborn,Chatelin} the eigenvalue
problem (\ref{weak_problem}) has eigenvalues
$$0<\lambda_{1,h}\leq \lambda_{2,h}\leq\cdots\leq \lambda_{k,h}\leq\cdots\leq \lambda_{N_h,h},$$
and the corresponding eigenfunctions
$$u_{1,h}, u_{2,h},\cdots, u_{k,h}, \cdots, u_{N_h,h},$$
where $b(u_{i,h},u_{j,h})=\delta_{ij}, 1\leq i,j\leq N_h$ ($N_h$ is
the dimension of the finite element space $V_h$).

From the minimum-maximum principle (\cite{Babuska2,BabuskaOsborn}),
the following upper bound result holds
$$\lambda_i\leq \lambda_{i,h}, \ \ \ i=1,2,\cdots, N_h.$$
Let $M(\lambda_i)$ denote the eigenspace corresponding to the
eigenvalue $\lambda_i$ which is defined by
\begin{eqnarray}
M(\lambda_i)&=&\big\{w\in V: w\ {\rm is\ an\ eigenvalue\ of\
(\ref{weak_problem})\ corresponding\ to} \ \lambda_i\nonumber\\
&&\ \ \ {\rm and}\ \|w\|_b=1\big\}.
\end{eqnarray}
Then we define
\begin{eqnarray}
\delta_h(\lambda_i)=\sup_{w\in M(\lambda_i)}\inf_{v\in
V_h}\|w-v\|_a.
\end{eqnarray}

For the eigenpair approximations by finite element method, there
exist the following error estimates.
\begin{proposition}(\cite[Lemma 3.7, (3.29b)]{Babuska2},\cite[P. 699]{BabuskaOsborn} and
\cite{Chatelin})\label{Error_estimate_Proposition}

\noindent(i) For any eigenfunction approximation $u_{i,h}$ of
(\ref{weak_problem_Discrete}) $(i = 1, 2, \cdots, N_h)$, there is an
eigenfunction $u_i$ of (\ref{weak_problem}) corresponding to
$\lambda_i$ such that $\|u_i\|_b = 1$ and
\begin{eqnarray}\label{Eigenfunction_Error}
\|u_i-u_{i,h}\|_a&\leq& C_i\delta_h(\lambda_i).
\end{eqnarray}
Furthermore,
\begin{eqnarray}\label{Eigenfunction_Error_Nagative}
\|u_i- u_{i,h}\|_{-a} \leq C_i\eta_a(h)\|u_i - u_{i,h}\|_a.
\end{eqnarray}
(ii) For each eigenvalue, we have
\begin{eqnarray}
\lambda_i \leq \lambda_{i,h}\leq \lambda_i +
C_i\delta_h^2(\lambda_i)
\end{eqnarray}
 Here and hereafter $C_i$ is some constant depending on $i$ but independent of  the mesh size $h$.
\end{proposition}

\section{One correction step}
In this section, we present a type of correction step to improve the
accuracy of the current eigenvalue and eigenfunction approximations.
This correction method contains solving an auxiliary source problem
in the finer finite element space and an eigenvalue problem on the
coarsest finite element space. For simplicity of notation, we set
$(\lambda,u)=(\lambda_i,u_i)\ (i=1,2,\cdots,k,\cdots)$  and
$(\lambda_h, u_h)=(\lambda_{i,h},u_{i,h})\ (i=1,2,\cdots,N_h)$ to
denote an eigenpair of problem (\ref{weak_problem}) and
(\ref{weak_problem_Discrete}), respectively.

To analyze our method, we introduce the error expansion of
eigenvalue by the Rayleigh quotient formula which comes from
\cite{Babuska2,BabuskaOsborn,LinYan,XuZhou}.
\begin{theorem}\label{Rayleigh_Quotient_error_theorem}
Assume $(\lambda,u)$ is the true solution of the eigenvalue problem
(\ref{weak_problem}),  $0\neq \psi\in V$. Let us define
\begin{eqnarray}\label{rayleighw}
\widehat{\lambda}=\frac{a(\psi,\psi)}{b(\psi,\psi)}.
\end{eqnarray}
Then we have
\begin{eqnarray}\label{rayexpan}
\widehat{\lambda}-\lambda
&=&\frac{a(u-\psi,u-\psi)}{b(\psi,\psi)}-\lambda
\frac{b(u-\psi,u-\psi)}{b(\psi,\psi)}.
\end{eqnarray}
\end{theorem}
\begin{proof}
First from (\ref{Rayleigh_quotient}), (\ref{rayleighw}) and direct
computation, we have
\begin{eqnarray}\label{expansion_lambad_lambda_1}
\widehat{\lambda}-\lambda&=&\frac{a(\psi,\psi)-\lambda
b(\psi,\psi)}{b(\psi,\psi)}\nonumber\\
&=&\frac{a(\psi- u,\psi- u)+2a(\psi,u)-a(u, u)-\lambda
b(\psi,\psi)}{b(\psi,\psi)}\nonumber\\
&=&\frac{a(\psi-u,\psi-u)+2\lambda b(\psi,u)-\lambda b(u,u)-\lambda b(\psi,\psi)}{b(\psi,\psi)}\nonumber\\
&=&\frac{a(\psi- u,\psi-u)-\lambda b(\psi-u,\psi-u)}{b(\psi,\psi)}.
\end{eqnarray}
Then we obtain the desired result (\ref{rayexpan}).
\end{proof}
Assume we have obtained an eigenpair approximation
$(\lambda_{h_1},u_{h_1})\in\mathcal{R}\times V_{h_1}$. Now we
introduce a type of correction step to improve the accuracy of the
current eigenpair approximation $(\lambda_{h_1},u_{h_1})$. Let
$V_{h_2}\subset V$ be a finer finite element space such that
$V_{h_1}\subset V_{h_2}$. Based on this finer finite element space,
we define the following correction step.

\begin{algorithm}\label{Correction_Step}
One Correction Step

\begin{enumerate}
\item Define the following auxiliary source problem:

Find $\tilde{u}_{h_2}\in V_{h_2}$ such that
\begin{eqnarray}\label{aux_problem}
a(\widetilde{u}_{h_2},v_{h_2})&=&\lambda_{h_1}b(u_{h_1},v_{h_2}),\ \
\ \forall v_{h_2}\in V_{h_2}.
\end{eqnarray}
Solve this equation to obtain a new eigenfunction approximation
$\tilde{u}_{h_2}\in V_{h_2}$.
\item  Define a new finite element
space $V_{H,h_2}=V_H+{\rm span}\{\widetilde{u}_{h_2}\}$ and solve
the following eigenvalue problem:

Find $(\lambda_{h_2},u_{h_2})\in\mathcal{R}\times V_{H,h_2}$ such
that $b(u_{h_2},u_{h_2})=1$ and
\begin{eqnarray}\label{Eigen_Augment_Problem}
a(u_{h_2},v_{H,h_2})&=&\lambda_{h_2} b(u_{h_2},v_{H,h_2}),\ \ \
\forall v_{H,h_2}\in V_{H,h_2}.
\end{eqnarray}
\end{enumerate}
Define $(\lambda_{h_2},u_{h_2})={\it
Correction}(V_H,\lambda_{h_1},u_{h_1},V_{h_2})$.
\end{algorithm}
\begin{theorem}\label{Error_Estimate_One_Correction_Theorem}
Assume the current eigenpair approximation
$(\lambda_{h_1},u_{h_1})\in\mathcal{R}\times V_{h_1}$ has the
following error estimates
\begin{eqnarray}
\|u-u_{h_1}\|_a &\lesssim &\varepsilon_{h_1}(\lambda),\label{Error_u_h_1}\\
\|u-u_{h_1}\|_{-a}&\lesssim&\eta_a(H)\|u-u_{h_1}\|_a,\label{Error_u_h_1_nagative}\\
|\lambda-\lambda_{h_1}|&\lesssim&\varepsilon_{h_1}^2(\lambda).\label{Error_Eigenvalue_h_1}
\end{eqnarray}
Then after one correction step, the resultant approximation
$(\lambda_{h_2},u_{h_2})\in\mathcal{R}\times V_{h_2}$ has the
following error estimates
\begin{eqnarray}
\|u-u_{h_2}\|_a &\lesssim &\varepsilon_{h_2}(\lambda),\label{Estimate_u_u_h_2}\\
\|u-u_{h_2}\|_{-a}&\lesssim&\eta_a(H)\|u-u_{h_2}\|_a,\label{Estimate_u_h_2_Nagative}\\
|\lambda-\lambda_{h_2}|&\lesssim&\varepsilon_{h_2}^2(\lambda),\label{Estimate_lambda_lambda_h_2}
\end{eqnarray}
where
$\varepsilon_{h_2}(\lambda):=\eta_a(H)\varepsilon_{h_1}(\lambda)+\varepsilon_{h_1}^2(\lambda)+\delta_{h_2}(\lambda)$.
\end{theorem}
\begin{proof}
From problems (\ref{Projection_Problem}), (\ref{weak_problem}) and
(\ref{aux_problem}), and (\ref{Error_u_h_1}),
(\ref{Error_u_h_1_nagative}) and (\ref{Error_Eigenvalue_h_1}), the
following estimate holds
\begin{eqnarray*}
\|\widetilde{u}_{h_2}-P_{h_2}u\|_a^2&\lesssim
&a(\widetilde{u}_{h_2}-P_{h_2}u,\widetilde{u}_{h_2}-P_{h_2}u)=b(\lambda_{h_1}u_{h_1}-\lambda
u,\widetilde{u}_{h_2}-P_{h_2}u)\nonumber\\
& \lesssim &\|\lambda_{h_1}u_{h_1}-\lambda
u\|_{-a}\|\widetilde{u}_{h_2}-P_{h_2}u\|_a\nonumber\\
&\lesssim & (|\lambda_{h_1}-\lambda|\|u_{h_1}\|_{-a}+\lambda
\|u_{h_1}-u\|_{-a})\|\widetilde{u}_{h_2}-P_{h_2}u\|_a\nonumber\\
&\lesssim &\big
(\varepsilon_{h_1}^2(\lambda)+\eta_a(H)\varepsilon_{h_1}(\lambda)\big)\|\widetilde{u}_{h_2}-P_{h_2}u\|_a.
\end{eqnarray*}
Then we have
\begin{eqnarray}\label{Estimate_u_tilde_u_h_2}
\|\widetilde{u}_{h_2}-P_{h_2}u\|_a &\lesssim
&\varepsilon_{h_1}^2(\lambda)+\eta_a(H)\varepsilon_{h_1}(\lambda).
\end{eqnarray}
Combining (\ref{Estimate_u_tilde_u_h_2}) and
 the error estimate of finite element projection
\begin{eqnarray*}
\|u-P_{h_2}u\|_a &\lesssim&\delta_{h_2}(\lambda),
\end{eqnarray*}
we have
\begin{eqnarray}\label{Error_tilde_u_h_2_u}
\|\widetilde{u}_{h_2}-u\|_{a}&\lesssim&\varepsilon_{h_1}^2(\lambda)+\eta_a(H)\varepsilon_{h_1}(\lambda)
+\delta_{h_2}(\lambda).
\end{eqnarray}
Now we come to estimate the eigenpair solution
$(\lambda_{h_2},u_{h_2})$ of problem (\ref{Eigen_Augment_Problem}).
Based on the error estimate theory of eigenvalue problem by finite
element method (\cite{Babuska2,BabuskaOsborn}), the following
estimates hold
\begin{eqnarray}\label{Error_u_u_h_2}
\|u-u_{h_2}\|_a&\lesssim& \sup_{w\in M(\lambda)}\inf_{v\in
V_{H,h_2}}\|w-v\|_a\lesssim \|u-\widetilde{u}_{h_2}\|_a,
\end{eqnarray}
and
\begin{eqnarray}\label{Error_u_u_h_2_Negative}
\|u-u_{h_2}\|_{-a}&\lesssim&\widetilde{\eta}_a(H)\|u-u_{h_2}\|_a,
\end{eqnarray}
where
\begin{eqnarray}\label{Eta_a_h_2}
\widetilde{\eta}_a(H)&=&\sup_{f\in V,\|f\|_a=1}\inf_{v\in
V_{H,h_2}}\|Tf-v\|_a \leq \eta_a(H).
\end{eqnarray}
From (\ref{Error_tilde_u_h_2_u}), (\ref{Error_u_u_h_2}),
(\ref{Error_u_u_h_2_Negative}) and (\ref{Eta_a_h_2}), we can obtain
(\ref{Estimate_u_u_h_2}) and (\ref{Estimate_u_h_2_Nagative}). The
estimate (\ref{Estimate_lambda_lambda_h_2}) can be derived by
Theorem \ref{Rayleigh_Quotient_error_theorem} and
(\ref{Estimate_u_u_h_2}).
\end{proof}

\section{Multi-level correction scheme}
In this section, we introduce a type of multi-level correction
scheme based on the {\it One Correction Step} defined in Algorithm
\ref{Correction_Step}. This type of correction method can improve
the convergence order after each correction step which is different
from the two-grid method in \cite{XuZhou}.

\begin{algorithm}\label{Multi_Correction}
Multi-level Correction Scheme
\begin{enumerate}
\item Construct a coarse finite element space $V_H$ and solve the
following eigenvalue problem:

Find $(\lambda_H,u_H)\in \mathcal{R}\times V_H$ such that
$b(u_H,u_H)=1$ and
\begin{eqnarray}\label{Initial_Eigen_Problem}
a(u_H,v_H)&=&\lambda_Hb(u_H,v_H),\ \ \ \ \forall v_H\in V_H.
\end{eqnarray}
\item Set $h_1=H$ and construct a series of finer finite element
spaces $V_{h_2},\cdots,V_{h_n}$ such that $\eta_a(H)\gtrsim
\delta_{h_1}(\lambda)\geq \delta_{h_2}(\lambda)\geq\cdots\geq
\delta_{h_n}(\lambda)$.
\item Do $k=0,1,\cdots,n-2$\\
Obtain a new eigenpair approximation
$(\lambda_{h_{k+1}},u_{h_{k+1}})\in \mathcal{R}\times V_{h_{k+1}}$
by a correction step
\begin{eqnarray}
(\lambda_{h_{k+1}},u_{h_{k+1}})=Correction(V_H,\lambda_{h_k},u_{h_k},V_{h_{k+1}}).
\end{eqnarray}
end Do
\item Solve the following source problem:

Find $u_{h_n}\in V_{h_n}$ such that
\begin{eqnarray}\label{aux_problem_h_n}
a(u_{h_n},v_{h_n})&=&\lambda_{h_{n-1}}b(v_{h_{n-1}},v_{h_n}),\ \ \
\forall v_{h_n}\in V_{h_n}.
\end{eqnarray}
Then compute the Rayleigh quotient of $u_{h_n}$
\begin{eqnarray}\label{Rayleigh_quotient_h_n}
\lambda_{h_n} &=&\frac{a(u_{h_n},u_{h_n})}{b(u_{h_n},u_{h_n})}.
\end{eqnarray}
\end{enumerate}
Finally, we obtain an eigenpair approximation
$(\lambda_{h_n},u_{h_n})\in \mathcal{R}\times V_{h_n}$.
\end{algorithm}
\begin{theorem}
After implementing Algorithm \ref{Multi_Correction}, the resultant
eigenpair approximation $(\lambda_{h_n},u_{h_n})$ has the following
error estimate
\begin{eqnarray}
\|u_{h_n}-u\|_a &\lesssim&\varepsilon_{h_n}(\lambda),\label{Multi_Correction_Err_fun}\\
|\lambda_{h_n}-\lambda|&\lesssim&\varepsilon_{h_n}^2(\lambda),\label{Multi_Correction_Err_eigen}
\end{eqnarray}
where
$\varepsilon_{h_n}(\lambda)=\sum\limits_{k=1}^{n}\eta_a(H)^{n-k}\delta_{h_k}(\lambda)$.
\end{theorem}
\begin{proof}
From $\eta_a(H)\gtrsim\delta_{h_1}(\lambda)\geq
\delta_{h_2}(\lambda)\geq\cdots\geq \delta_{h_n}(\lambda)$ and
Theorem \ref{Error_Estimate_One_Correction_Theorem}, we have
\begin{eqnarray}
\varepsilon_{h_{k+1}}(\lambda)
&\lesssim&\eta_a(H)\varepsilon_{h_k}(\lambda)+\delta_{h_{k+1}}(\lambda),\
\ \ {\rm for}\ 1\leq k\leq n-2.
\end{eqnarray}
Then by recursive relation, we can obtain
\begin{eqnarray}\label{epsilon_n_1}
\varepsilon_{h_{n-1}}(\lambda)&\lesssim&\eta_a(H)\varepsilon_{h_{n-2}}(\lambda)
+\delta_{h_{n-1}}(\lambda)\nonumber\\
&\lesssim&\eta_a(H)^2\varepsilon_{h_{n-3}}(\lambda)+
\eta_a(H)\delta_{h_{n-2}}(\lambda)+\delta_{h_{n-1}}(\lambda)\nonumber\\
&\lesssim&\sum\limits_{k=1}^{n-1}\eta_a(H)^{n-k-1}\delta_{h_k}(\lambda).
\end{eqnarray}
Based on the proof in Theorem
\ref{Error_Estimate_One_Correction_Theorem} and (\ref{epsilon_n_1}),
the final eigenfunction approximation $u_{h_n}$ has the error
estimate
\begin{eqnarray}\label{Error_u_h_n_Multi_Correction}
\|u_{h_n}-u\|_a&\lesssim&\varepsilon_{h_{n-1}}^2(\lambda)+\eta_a(H)\varepsilon_{h_{n-1}}(\lambda)
+\delta_{h_n}(\lambda)\nonumber\\
&\lesssim& \sum_{k=1}^n\eta_a(H)^{n-k}\delta_{h_k}(\lambda).
\end{eqnarray}
This is the estimate (\ref{Multi_Correction_Err_fun}). From Theorem
\ref{Rayleigh_Quotient_error_theorem} and
(\ref{Error_u_h_n_Multi_Correction}), we can obtain the estimate
(\ref{Multi_Correction_Err_eigen}).
\end{proof}

\section{The application to second order elliptic eigenvalue problem}
In this section, for example, the multi-level correction method
presented in this paper is applied to the second order elliptic
eigenvalue problem. We also discuss two possible ways to implement
the multi-level correction Algorithm \ref{Multi_Correction}. The
first way is the ``two-grid method" of Xu and Zhou introduced and
studied in \cite{XuZhou}. The second one proposed and studied by
Andreev and Racheva in \cite{AndreevLazarovRacheva,RachevaAndreev}
uses the same mesh but higher order finite elements.

In (\ref{weak_problem}), the second order elliptic eigenvalue
problem can be defined by
\begin{eqnarray*}
a(u,v)=\int_{\Omega}\nabla u\cdot\mathcal{A}\nabla vd\Omega, &&
b(u,v)=\int_{\Omega}\rho uvd\Omega,
\end{eqnarray*}
where $\Omega\subset\mathcal{R}^d \ (d=2,3)$ is a bounded domain,
$\mathcal{A}\in \big(W^{1,\infty}(\Omega)\big)^{d\times d}$  a
uniformly positive definite matrix on $\Omega$ and $\rho\in
W^{0,\infty}(\Omega)$ is a uniformly positive function on $\Omega$.
We pose Dirichlet boundary condition to the problem and it means
here $V=H_0^1(\Omega)$ and $W=L^2(\Omega)$. In order to use the
finite element discretization method, we employ the meshes defined
in section 3.

Here, we introduce two ways to implement the multi-level correction
Algorithm \ref{Multi_Correction}. The first way uses finer meshes to
construct the series of finite element spaces. The advantage of this
approach is that it uses the same finite element and does not
require higher regularity of the exact eigenfunctions
(\cite{RachevaAndreev}). The second way is based on the same finite
element mesh but using higher order finite elements. In order to
improve the convergence order, the higher regularity of the exact
eigenfunctions is required.

Let us discuss the methods to construct the series of finite element
spaces $V_{h_k} (k=2,3,\cdots,n)$ for implementing the multi-level
correction method.

{\it Way 1.} (``Multi-grid method"): In this case, $V_{h_k}\
(k=2,3,\cdots,n)$ is the same type of finite element as $ V_{H}$ on
the finer mesh $\mathcal{T}_{h_k}$ with smaller mesh size $h_{k}$.
Here $\mathcal{T}_{h_k}$ is a finer mesh of $\Omega$ which can be
generated by the refinement just as in the multigrid method
(\cite{XuZhou}) from $\mathcal{T}_{h_{k-1}}$ such that
$h_{k}=\eta_a(H)h_{k-1}$. Assume the computing domain $\Omega$ is a
convex domain. Then $\eta_a(H)=\mathcal{O}(H)$ and
$\delta_{h_k}=\mathcal{O}(h_k)=\mathcal{O}(H^k)\ (k=1,2,\cdots, n)$,
and we can obtain the following error estimate for
$(\lambda_{h_n},u_{h_n})$
\begin{eqnarray}
|\lambda-\lambda_{h_n}|&\lesssim&\eta_a(H)^{2n-2}\delta_{H}^2(\lambda)=\mathcal{O}(H^{2n}),\\
\|u-u_{h_n}\|_a&\lesssim&\eta_a(H)^{n-1}\delta_{H}(\lambda)=\mathcal{O}(H^n).
\end{eqnarray}
From the error estimates above, we can find that the multi-level
correction scheme can obtain the accuracy as same as solving the
eigenvalue problem on the finest mesh $\mathcal{T}_{h_n}$. This
improvement costs solving the source problems on the finer finite
element spaces $V_{h_k}\ (k=2,3,\cdots,n)$ and the eigenvalue
problems in coarse spaces $V_{H,h_k}\ (k=2,3,\cdots,n-1)$. This is
better than solving the eigenvalue problem on the finest finite
element space directly, because solving source problem needs much
less computation than solving the corresponding eigenvalue problem.

{\it Way 2.} (``Multi-space method"): In this case, $V_{h_k}$ is
defined on the same mesh $\mathcal{T}_H$ but using higher order
finite element than $V_{h_{k-1}}$. In order to describe the scheme
simply, we suppose the exact eigenfunction has sufficient
regularity. We use the linear finite element space to solve the
original eigenvalue problem (\ref{weak_problem}) on $V_H$, and solve
the source problem (\ref{aux_problem}) in higher order finite
element space with the way that the order of $V_{h_k}$ is one order
higher than $V_{h_{k-1}}$. Then we have the following error estimate
for the final eigenpair approximation $(\lambda_{h_n}, u_{h_n})$
\begin{eqnarray}
|\lambda-\lambda_{h_n}|&\lesssim&\eta_a(H)^{2n-2}\delta_{H}^2(\lambda)=\mathcal{O}(H^{2n}),\\
\|u-u_{h_n}\|_a&\lesssim&\eta_a(H)^{n-1}\delta_H(\lambda)=\mathcal{O}(H^{n}).
\end{eqnarray}
The improved error estimates above just cost solving the source
problems on the same mesh but in higher order finite element spaces
and eigenvalue problems in the lowest order finite element space.

\section{Numerical results}
In this section, we give two numerical examples to illustrate the
efficiency of the multi-level correction algorithm proposed in this
paper. We solve the model eigenvalue problem (\ref{problem}) on the
unit square  $\Omega=(0,1)\times(0,1)$.

\subsection{Multi-space way}
Here we give the numerical results of the multi-level correction
scheme in which the finer finite element spaces are constructed by
improving the finite element orders on the same mesh. We first solve
the eigenvalue problem (\ref{problem}) in linear finite element
space on the mesh $\mathcal{T}_H$. Then do the first correction step
with quadratic element and cubic element for the second one.

Here, we adopt the meshes which are produced by regular refinement
from the initial mesh generated by Delaunay method to investigate
the convergence behaviors. Figure \ref{Initial_Mesh} shows the
initial mesh. Figure \ref{numerical_multi_space_1} 
gives the corresponding numerical results for the first eigenvalue
$\lambda_1=2\pi^2$ and the corresponding eigenfunction. 
\begin{figure}[ht]
\centering
\includegraphics[width=6cm,height=6cm]{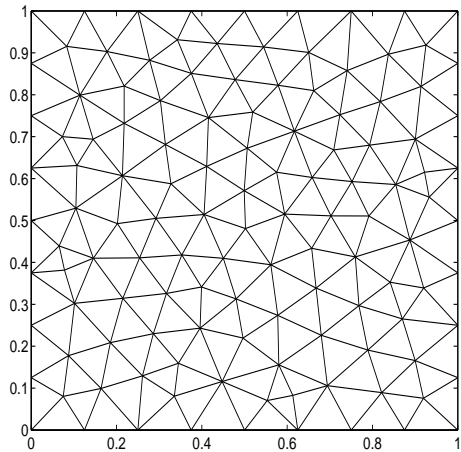}
\caption{Initial mesh for multi-space way} \label{Initial_Mesh}
\end{figure}
\begin{figure}[ht]
\centering
\includegraphics[width=6cm,height=7cm]{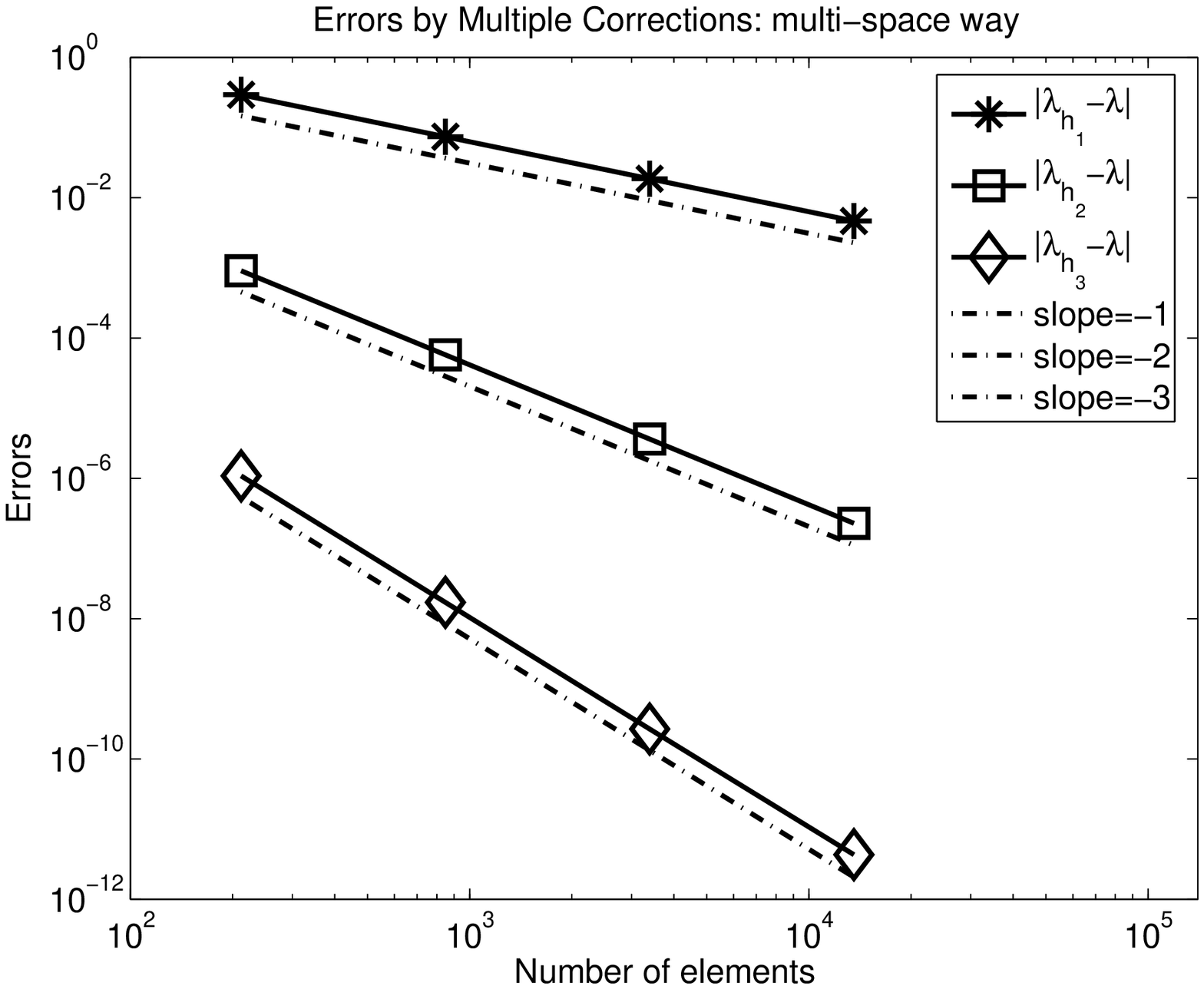}
\includegraphics[width=6cm,height=7cm]{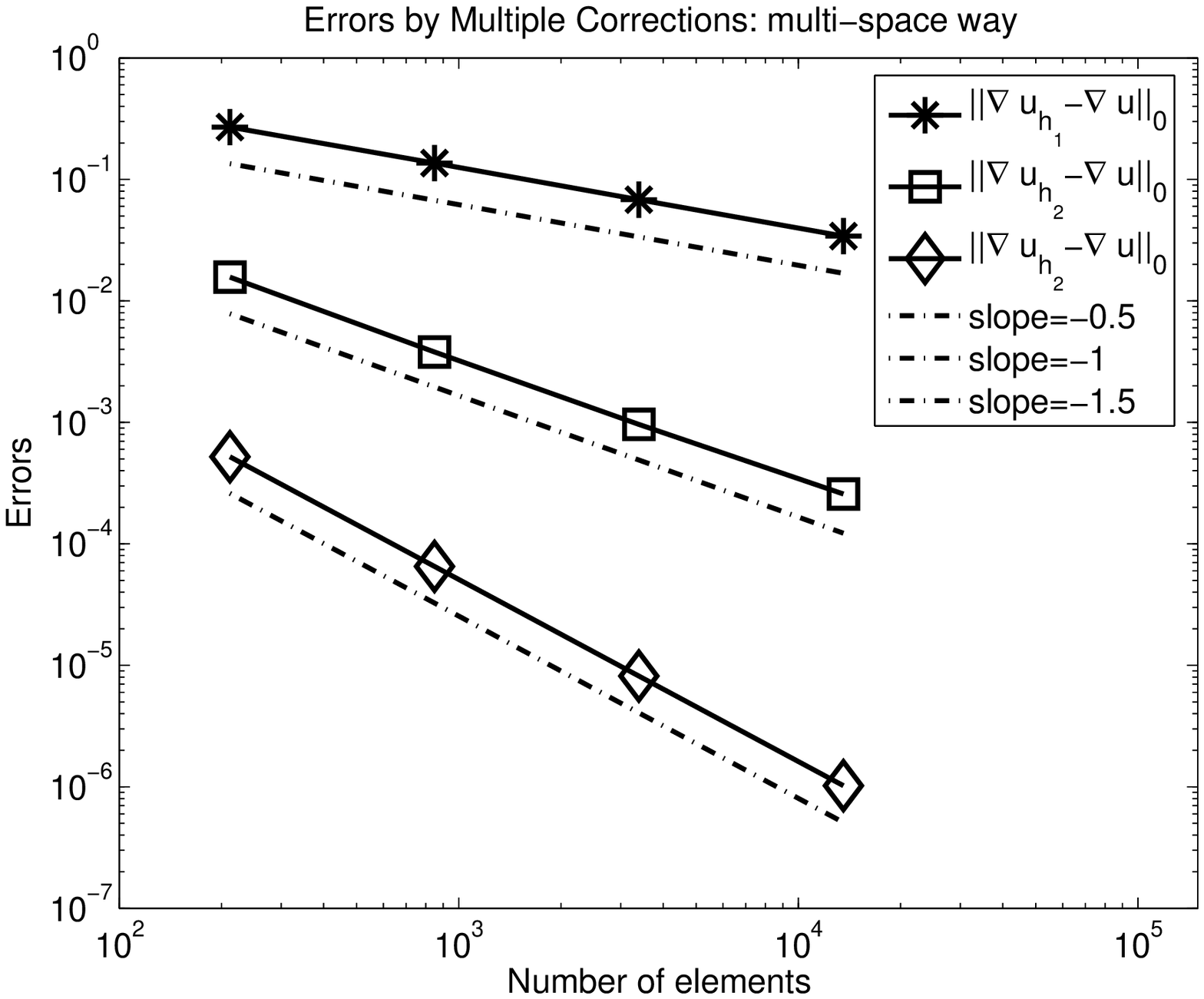}
\caption{The errors for the eigenvalue approximations by multi-level
correction algorithm for the first eigenvalue $2\pi^2$ and the
corresponding eigenfunction with multi-space way}
\label{numerical_multi_space_1}
\end{figure}
From Figure \ref{numerical_multi_space_1},
 we can find each correction step can improve the convergence order by two for
 the eigenvalue approximations and one for the eigenfunction approximations
  with multi-space way when the exact eigenfunction is smooth.

\subsection{Multi-grid way}
Here we give the numerical results of the multi-level correction
scheme where the finer finite element spaces are constructed by
refining the existed mesh. We first solve the eigenvalue problem
(\ref{problem}) by linear finite element space on the mesh
$\mathcal{T}_H$. Then refine the mesh by the regular way such that
the size of the resultant mesh $h_k=O(H^k)$ to obtain the mesh
$\mathcal{T}_{h_k}\ (k=2,\cdots,n)$ and solve the auxiliary source
problem (\ref{aux_problem}) in the finer linear finite element space
$V_{h_k}$ defined on $\mathcal{T}_{h_k}$ and the corresponding
eigenvalue problem in $V_{H,h_k}$.

Figure \ref{numerical_multi_grid_1} gives the corresponding
numerical results for the first eigenvalue $\lambda=2\pi^2$ and the
corresponding eigenfunction on the uniform meshes.
\begin{figure}[ht]
\centering
\includegraphics[width=6cm,height=7cm]{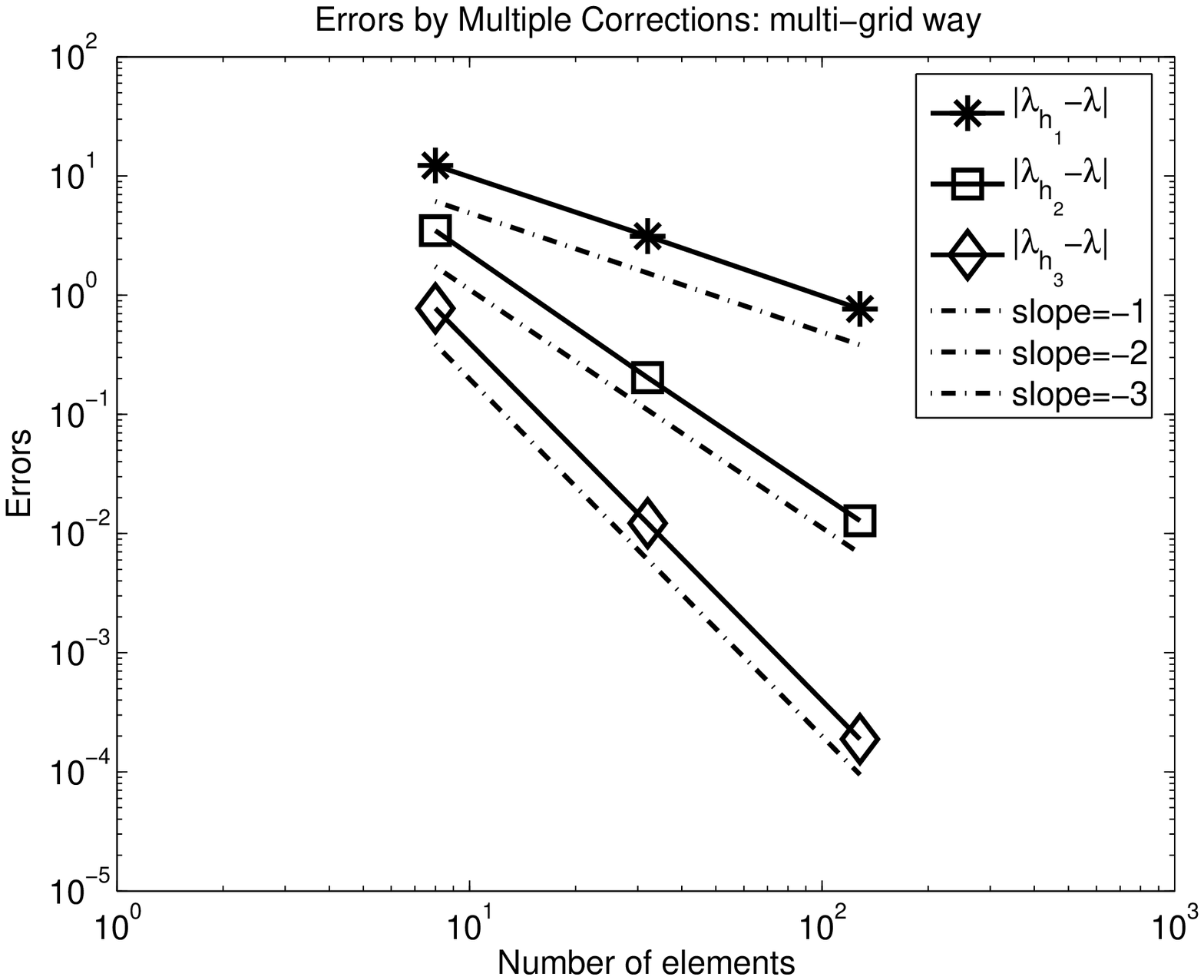}
\includegraphics[width=6cm,height=7cm]{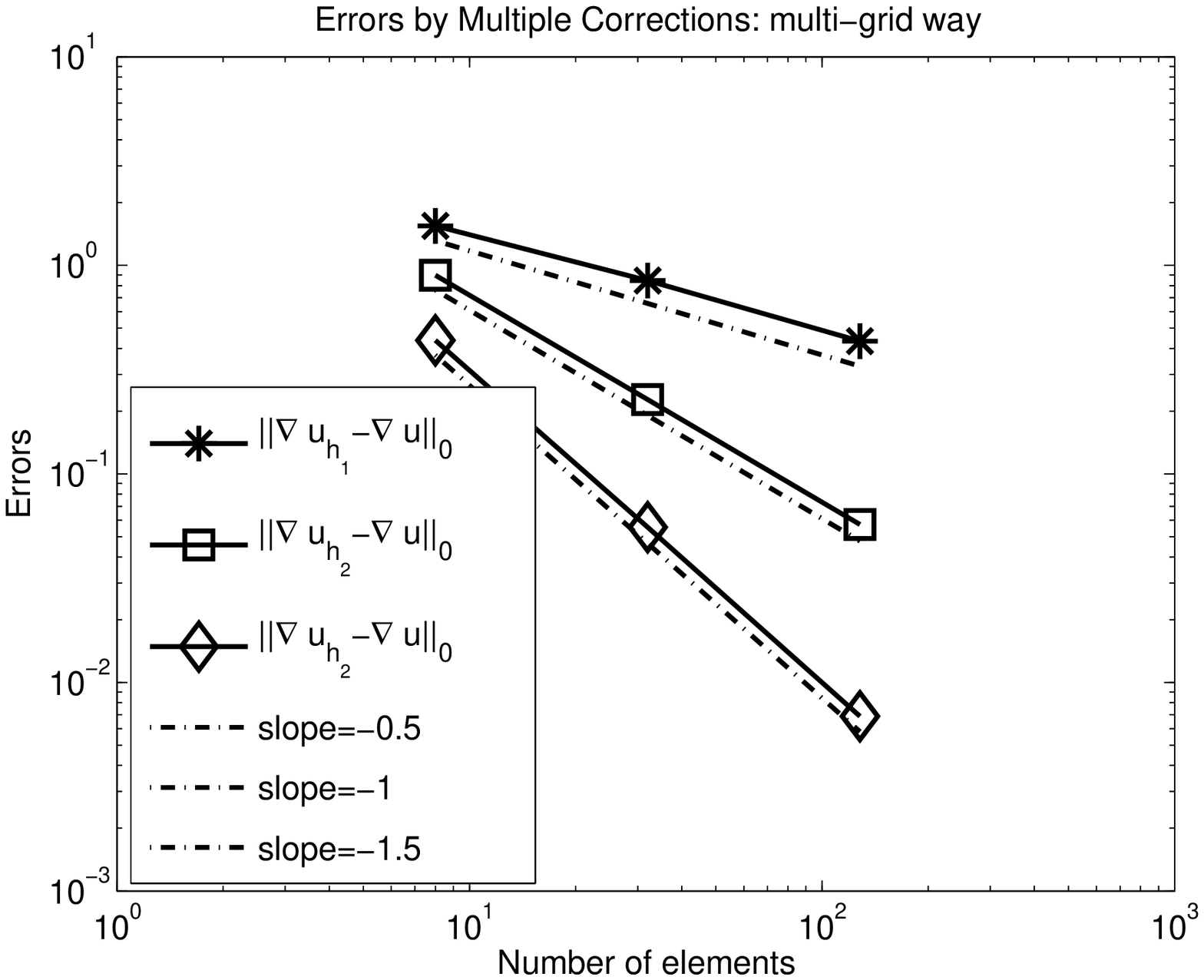}
\caption{The errors for the eigenvalue approximations by multi-level
correction algorithm for the first eigenvalue $2\pi^2$ with
multi-grid way} \label{numerical_multi_grid_1}
\end{figure}
From Figure \ref{numerical_multi_grid_1}, we can also find each
correction step can improve the convergence order by two for the
eigenvalue approximations and one for the eigenfunction
approximations with the multi-grid way.

\section{Concluding remarks}
In this paper, we give a new type of multi-level correction scheme
to improve the accuracy of the eigenpair approximations.  We can use
the better eigenvalue and eigenfunction approximation
$(\lambda_{h_n}, u_{h_n})$ to construct an a posteriori error
estimator of the eigenpair approximation for the eigenvalue problem
(\cite{ChenLin,LinXie}).

Furthermore, our multi-level correction scheme can be coupled with
the multigrid method to construct a type of multigrid and parallel
method for eigenvalue problems (\cite{XuZhou_Eigen}). It can also be
combined with the adaptive refinement technique for the singular
 eigenfunction cases. These will be our future work.

\bibliographystyle{amsplain}

\end{document}